\documentclass[12pt]{article}
\usepackage{amsmath}
\usepackage{amsthm}
\usepackage{graphicx}
\usepackage{amssymb}
\usepackage{amsfonts}
\usepackage{romannum}
\usepackage[toc,page]{appendix}
\usepackage[final]{changes}
\newtheorem{thm}{Theorem}
\newtheorem{lem}{Lemma}

\newtheorem{prop}{Proposition}

\begin{document}

\setcounter{section}{0} 
\pagenumbering{arabic}
\setcounter{page}{1}
\newtheorem{lemma}{Lemma}
\newtheorem{theorem}{Theorem}
\newtheorem{remark}{Remark}[section]
\newtheorem{corollary}{Corollary}[section]
\newtheorem{proposition}{Proposition}
\newcommand{\vect}[1]{\overline{#1}}
\newcommand{\dt}{\mathit{dt}}
\def\e{\eta}
\def\m{\mu}
\def\a{\alpha}
\def\r{\rho} \def\s{\sigma}
\def\l{\lambda}
\def\P{\mbox{P}}
\def\E{\mbox{\rm E}}
\def\Cov{\mbox{\rm Cov}}
\def\o{\omega} \date{}
\title{Customers' abandonment strategy in an $M/G/1$ queue}
\author{Eliran Sherzer and Yoav Kerner}

\newcommand{\overbar}[1]{\mkern 1.5mu\overline{\mkern-1.5mu#1\mkern-1.5mu}\mkern 1.5mu}

\maketitle

\begin{abstract}
We consider an $M/G/1$ queue in which the customers, while waiting
in line, may renege from it. We study the Nash equilibrium profile
among customers, and show that it is defined by two sequences of
thresholds. For each customer, the decision is based on the observed
past (which determines from what sequence the threshold is taken),
and the observed queue length (which determines the appropriate
element in the chosen sequence). We construct the a of equations
that has the Nash equilibrium as its solution, and discuss the
relationships between the properties of the service time
distribution and the properties of the Nash equilibrium, such as
uniqueness and finiteness.
\end{abstract}

\section{Introduction}

Understanding customers' abandonments from a queue is of interest
for service providers and customers. \replaced{There are many
applications concerning customers' abandonments since}{It has many
applications because} real-world customers are unwilling to wait for
excessive lengths of time. Applications are presented more broadly
by Mandelbaum and Shimkin~\cite{Mandelbaum00}.  Traditional queueing
theory has dealt with the analysis of queues under the assumption of
\added{a} given patience distribution\deleted{s}, and many studies
have addressed \replaced{models which include abandonments}{this
problem}, starting with Barrer~\cite{Barrer57}, who studied the
queue length distribution in the $M/M/s+D^w$ case (where $D$
indicate\replaced{s}{d} deterministic patience). Sufficient
conditions for the existence of the steady-state virtual waiting
time distribution in the $G/G/1+G^w$ were later obtained by
Baccelli~\cite{Baccelli84,Baccelli81}. Boxma et al.~\cite{boxma10}
showed how to determine the busy-period distribution for various
choices of the patience time distribution. Brandt and
Brandt~\cite{Brant13} also studied the distribution of the busy
period. In particular, they gave an explicit representation of the
Laplace-Stieltjes transform of the workload and the busy period, in
the case of phase-type distributed impatience.
Yet, these studies assumed that the patience distribution was given. Here, we investigate how these distributions are constructed by rational behavior of customers, which are affected by factors such as individual costs and preferences. Essentially, the patience of each customer is based on an individual optimization, that is, the perceived balance between the costs of waiting and the benefits of service. Since the behaviour of others has an influence on the individual (abandonments of others shorten the individual's required waiting time), the mutual interactions lead us to look at the system in the standard form of a game, and to seek the Nash equilibrium. Therefore, the patience distribution is no longer given, but instead it results from a cost/reward model, and from the strategic behaviour implied by it.\\
\indent The seminal study that viewed queues as economic systems,
and studied the strategic behaviors within them, came from
Naor~\cite{Naor69}. Naor considered an observable $M/M/1$ queue, in
which there is a constant reward $R$ from service, and a constant
waiting time rate $C$. He showed that under self-optimization,
customers will join the queue if the number of customers present
upon arrival is less than a threshold, $n$
$n=\lfloor{\frac{R\mu}{C}}\rfloor $.
 Moreover, once a customer joins the queue, he stays until served. Hassin and Haviv~\cite{Hassin95} also studied an $M/M/1$ queue,
  but in their case the customers have no information about their current position in the queue (i.e., unobservable queue).
   The \deleted{reward and} waiting cost \replaced{is}{are} the same as in Naor's study. However, the reward is $R>0$ if service is completed in less than $T$ time units, otherwise it equals $0$.
    They showed that the pure equilibrium strategy is as follows: join the queue with some probability $p$ and otherwise balk, and if one joined then he waits $T$ time units. The reason for this result is that the virtual waiting time follows an increasing failure rate (IFR) pattern. That is, the remaining waiting time stochastically decreases along with the time passed. This is true for every $M/M/1$ queue with impatient customers, as Baccelli and Hebuterne previously showed \added{in}~\cite{Baccelli81}.\\
    \indent Mandelbaum and Shimkin~\cite{Mandelbaum00} retained the assumptions of unobservability, linear waiting costs and constant service reward, but instead considered \added{an} $M/M/m$ queue\deleted{s} and heterogeneous customers whose waiting costs and rewards \replaced{may}{could} vary between \added{the} different customer types.
    Their results \replaced{indicate}{showed} that,
    depending on the reward to cost ratio,
     a customer's best response is to either abandon the queue upon arrival unless of course one of the servers is available,
      in which case he enters service immediately or, given the IFR property, to never abandon and wait until receiving service.
       Moreover, they proposed a case in which the solution is richer, in which customers are discharged without knowing it\added{,}
       (see\added{,} e.g.~\cite{Palm53}). In their\deleted{new} proposed model, each customer will never be served with probability $1-q$.
       Thus, the longer a customer has already waited, the higher the posterior probability that he has been discharged.
        It turns out that the system has an eventually decreasing (and in fact unimodal) hazard rate function, which makes finite abandonments rational.
         Thus, Mandelbaum and Shimkin showed that the best response is to either abandon at arrival, unless,
         of course, one of the servers is available or to abandon after a finite time $T$, which is determined by the ratio of the waiting cost and the
         reward. A follow up study by Mandelbaum and Shimkin~\cite{Mandelbaum04} considered a nonlinear waiting cost in an unobservable $M/M/m$ queue.
          They provided conditions for the existence and uniqueness of the equilibrium, and suggested procedures for its computation.\\
          \indent Another important
           study, by Haviv and Ritov~\cite{Haviv01}, explored an $M/M/1$ queue where the waiting cost is nonlinear and the reward from service is no
            longer constant. They showed that the equilibrium \added{strategy} is to abandon after waiting
             \replaced{$T$ time units in the queue}{a time $T$} ($T =0$ or $T=\infty $ are possible).
              Also, they showed that having a mixed Nash equilibrium may occur when the ratio of the cost and reward functions satisfie\replaced{s}{d}
               certain conditions. In this model, the waiting time \replaced{has an}{follows} IFR, and thus the remaining waiting time reduces with the
                elapsed waiting time. However, the waiting cost is convex, and therefore waiting becomes more expensive.
                 The latter balances the improvement in the waiting time and the remaining waiting
                 costs. \\
                 \indent Afeche and Sarhangian~\cite{Afeche15} also study
                 customers abandonments in an observable $M/M/1$ queue, where they consider pricing
                 as a means control the behavior of some of the
                 customers.
    Cripps and Thomas~\cite{Cripps14}, and Debo at al~\cite{Debo12}, considered a
    discrete-time model of an observable single server queue with homogenous customers who maximize discounted payoffs.
     In those studies, he server is bad (i.e. not
functioning) with a given probability and good otherwise. Both in
~\cite{Cripps14} and~\cite{Debo12}, study customers strategic
behavior and given what they observe in the queue the need to first
decide wether join or balk and then decide if and when to renege. A
very interesting result in the both studies that if one reneges,
then all the customers in worse positions in the queue renege as
well. Later we show that this phenomenon is also in our study. More
related work can found in~\cite{HH} and~\cite{Ha}.\\
 \indent So far, we
have introduced a variety of models dealing with customers'
abandonments from a queue. They vary in many ways, such as their
observability level, their cost and reward functions, and number of
servers. However, all of these studies had something in common, they
all assumed an exponential service
distribution\footnote{Except~\cite{Cripps14} and~\cite{Debo12}.
However, they focus on discrete-time model.}. In this study, we
\deleted{would like to}relax this assumption. In particular, we
consider an observable $M/G/1$ queue with a First-Come-First-Served
(FCFS) discipline. \replaced{With the term}{By} observable we mean
that everyone can see their own position in the queue at any time,
and they are able to observe service completions and abandonments
made by others. Unlike the case where the service distribution is
exponential, the \deleted{service} age in \deleted{the general} case
\added{of generally distributed service time} is meaningful.
\replaced{The age}{This} determines our anticipation of the
remaining time for the current service, and consequently the waiting
time. Hence, we allow customers to keep track of time.\\
\indent  When considering expected utility, we specify that
customers have a linear waiting cost and a constant reward from
service. Customers don't have a waiting cost while they are being
served. We assume that all customers are rational, in the sense that
\replaced{every customer}{each one} wishes to maximize his own
expected utility, and so they will stay as long as their expected
utility is positive. Another consequence that resulted from relaxing
the memoryless service distribution is the fact that the waiting
time doesn't necessarily have to be with IFR.
 As shown above, this played a major role in studies~\cite{Hassin95}-\cite{Naor69}.
  If the service distribution is with IFR, (which implies that the waiting time distribution is also with IFR),
  then the solution in trivial, in which, if one joined then he will stay forever. Thus, we focus on a case where the solution is richer and therefore, we limit ourselves to service distributions with
decreasing failure rate (DFR). That is, the remaining waiting time
stochastically increases \deleted{along} with \deleted{the} time
\deleted{passed}. In fact, in our model, we can use a weaker
property. Specifically, we require only that \added{the} service
time mean residual life (MRL) is increasing with \replaced{its
age}{its past}. Consequently, customers' expected queueing time
increases with the service age.\\
\indent An example for service distribution with DFR is the
following: Think of a queue for purchasing train tickets from a
machine. Suppose there are two kinds of customers, those who are
local and hence more experienced and on the other hand unexperienced
customers. An experience customer shouldn't delay more than a few
seconds for buying a ticket. Thus, if after a few seconds since the
service began a completion hasn't occurred, then it is highly
possible that an unexperienced customers is the one who buys the
ticket and hence the remaining service time is larger. We also note
that, if the service time MRL isn't bounded then rational customers
will surely abandon at some point, since their expected utility will
eventually be negative. An example of this kind \added{of
distribution} is the Pareto distribution. However, if the service
time MRL is bounded, it is possible that customers will be willing
to wait forever, given a big enough service reward. An
example of a distribution of this kind is the hyperexponential.\\
\indent The rest of the paper is arranged as follows. In
Section~\ref{section:model} we present the model, and customers'
expected utility and strategy profile. Next, in
Section~\ref{section:result} we give the theoretical framework with
instructions on how to obtain the Nash equilibrium. This is followed
by a numerical example. Finally, in Section~\ref{section:discuusion}
we discuss and summarize our primary results.

\section{Model formulation} \label{section:model}

\subsection{Model description}

We consider a single server queueing model in which the arrivals are
according to a Poisson process with rate $\l$ and service times
\replaced{are iid}{comes from a} generally distributed. We denote
the service time\deleted{s} by $X$, and $\mathbb{E}[X]$ by
$\bar{x}$. We also \added{use the notation}  \deleted{denote}
$f(\cdot)$ \replaced{for the}{to be a} pdf and $F(\cdot)$
\replaced{for the}{to be a} cdf, with $\bar{F}(\cdot)=1-F(\cdot)$,
and the hazard function is denoted by $h(\cdot)$.  The
\replaced{service}{queue} discipline is FCFS. Each customer can see
his position in the queue at any moment, and \added{he} is able to
observe service completions and abandonments made by others.
However, customers are not aware of events that occurred prior to
their arrival. Of course, they do not anticipate future service
times. Customers can keep track of time, and are allowed to abandon
at any moment. All customers are homonomous in their reward from
service, which is denoted by $V$, and a linear waiting cost, which
is denoted by $C$. In order to complete the model description we
first present the following definition. For a non-negative random
variable $X$ with \replaced{cdf}{distribution function} $F(\cdot)$,
\deleted{define} the MRL function \added{is defined as follows}:\\
$$m_X(x)=\mathbb{E}[X\!-\!x|X\!>\!x]=\int_{0}^{\infty}\frac{\overbar{F}(x+t)}{\overbar{F}(x)}dt,\quad
x\geq 0$$ Also, let $ m_X(\infty)=\underset{x\rightarrow
\infty}{\lim} m_X(x)$. We distinguish between service distributions
in which the MRL is bounded, and those in which it is unbounded.
\added{We assume that the} service time \replaced{has an}{follows}
increasing mean residual life (IMRL)~\cite{Balow}.

\subsection{Utility function}

For each individual the expected utility function balances
\added{the} reward from service with \added{the} expected waiting
cost. Customers always take into account their future costs, while
the time they already waited is considered a sunk cost. Each
customer wishes to maximize his own expected utility; as a result,
his best response is to stay as long as the expected utility is
positive. We first consider a case where abandonments are not
allowed. This will assist un contract the utility function with
abandonments. Let $G_n(t)$ be the expected utility function value
from staying until being served, for an individual that has $n$
customers in front of him in the system when the current service age
is $t$. Clearly,
 \begin{align*}
    G_n(t)=  \begin{cases}
              V \hspace{8cm}                &  n=0\\
              V-C(\mathbb{E}[X-t|X>t]+(n-1)\bar{x})\quad  &  n\geq1  \end{cases}
 \end{align*}
We next present a sequence of differential equations that can be
used to solve  $G_n(t)$ for $n\geq 1$.
\begin{prop}  \label{prop:G(t)}
For any value of $n\geq 0$, $G_n(t)$ solves the sequence of
differential equations for $n\geq 1$,
\begin{align}\label{eq:prop1}
G'_n(t)=C-h(t)(G_{n-1}(0)-G_n(t)) .
\end{align}
\end{prop}
\begin{proof}
First, one can see that $G_{n-1}(0)-G_n(t)= C\mathbb{E}[X-t|X>t]$
and hence the RHS of Equation~(\ref{eq:prop1}) equals
$C-Ch(t)\mathbb{E}[X-t|X>t]$. We also note that
$\mathbb{E}[X-t|X>t]=\frac{\int_{s=t}^{\infty}\bar{F}(s)ds}{\overbar{F}(t)}$.
Hence, the derivative of $G_n(t)$ is \begin{align*}
\frac{dG_n(t)}{dt}=-C\frac{-\overbar{F}(t)\overbar{F}(t)+f(t)\int_{s=t}^{\infty}\bar{F}(s)ds}{\overbar{F}(t)^2}=C-Ch(t)\mathbb{E}[X-t|X>t]
\end{align*}
which \replaced{matched}{coincides with} the derivative presented in
the proposition, and hence the proof is concluded.
\end{proof}
\begin{remark}
The intuition behind the proposition is as follows: Since there is a
linear waiting cost, then clearly, $C$ should be in the derivative.
Moreover, if service completion occurs, then the expected utility
changes from $G_n(t)$ to $G_{n-1}(0)$, which happens with rate $h(t)$.
\end{remark}
\noindent Deriving the expected utility in our model is
much harder than the one proposed in Proposition~\ref{prop:G(t)}. This is because, customers have to take into account the fact that
abandonments may occur, that is, one needs to consider the
possibility that he abandons later. Nonetheless, we next show that
our expected utility function is the positive part of $G_n(t)$. The
minor difference between the two functions shows how the possibility
of abandoning later reflects in the expected utility function. We
distinguish between customers who already observed a service
completion (called type \textbf{\romannum{1 }} customers) and those
who haven't observed a service completion (type \textbf{\romannum{2
}} customers) \footnote{Upon arrival, everyone are type
\textbf{\romannum{2 }} customers, which potentially may switch to
type \textbf{\romannum{1 }} due to being present in the queue while
service completion occurred.}. Due to the fact that keeping track of
time is possible, type \textbf{\romannum{1 }} customers know the
exact service age. However, since the current service age is unknown
upon arrival, type \textbf{\romannum{2 }} customers can only
estimate it. Let $U_n(t)$ be the \added{expected} utility of a type
\textbf{\romannum{1 }} customer taking the optimal action in the
next moment, given that there are $n$ others in front of him in the
system and $t$ time units elapsed since the last service completion.
\begin{prop}  \label{prop:u(t)}
    $$U_n(t) =  (G_n(t))^+, \quad   n \geq 0$$
\end{prop}
\begin{proof}
$n\!=\!0$ implies that an individual is already in service, and
hence has no cost. Clearly, the optimal action in the next moment is
to stay. For $n\geq 1$, we first prove the proposition for $n=1$ and
then prove for $n\geq2$ by induction. For $n=1$, we have
\begin{align}
U_1(t)=(h(t)dtU_{0}(0)+(1-h(t)dt)U_1(t+dt)-Cdt+o(dt))^+, \quad t
\geq 0 \label{eq:u1t}.
\end{align}
If the RHS of~(\ref{eq:u1t}) is negative, then the best response is
to abandon, and hence the expected utility is zero. If the RHS is
positive, then we have the following differential equation:
\begin{align*}
U_1'(t)=C-h(t)(U_{0}(0)-U_1(t))
\end{align*}
where clearly, $U_{0}(0)=V$. This differential equation coincides
with Proposition~\ref{prop:G(t)}, and hence $U_1(t)=
(V-C(\mathbb{E}[X-t|X>t]+\bar{x}))^+$. We first assume that
$U_{n-1}(t)=(G_{n-1}(t))^+$. Based on our assumption, clearly,
$U_{n-1}(0)=(V-C(n-1)\bar{x})^+$. Combining with
\begin{align*}
U_n(t)=h(t)dtU_{n-1}(0)+(1-h(t)dt)U_n(t+dt)-Cdt+o(dt)
\end{align*}
therefore
\begin{align*}
U_n'(t)=C-h(t)(U_{n-1}(0)-U_n(t))
\end{align*}
and solving the differential equation \deleted{we} complete
\replaced{the}{our} proof.
\begin{remark}
One can see that the only way the possibility of abandoning later
reflects in the expected utility function is by abandoning once the
expected utility becomes negative. Consequently, only myopic
decisions are under consideration.
\end{remark}
\end{proof}
\noindent Let $\hat{U}_n(t)$ be the expected utility of a type
\textbf{\romannum{2 }} customer taking the optimal action in the
next moment, given \added{that} $n$ other in front of him in the
system and $t$ time units elapsed since his arrival. Formally, let
$N$ be the number of customers \added{in the system} and $A(t)$ be
the service age, both \replaced{upon}{at} arrival\deleted{epochs}.
Finally, let $X_n(t)$ follow the distribution of the service
residual, given that upon the arrival of a tagged customer there
were $n$ customers in the system and $t$ time units elapsed since
then. That is, $X_n(t)\,{\buildrel d \over =}\
\{X-\!A(t)|N\!=\!n\}$.
\begin{prop}
\begin{align*}
    \hat{U}_n(t) =
              (V-C(\mathbb{E}[X_n(t)-t|X_n(t)>t]+(n-1)\bar{x}))^+\quad    n \geq1
\end{align*}
\label{prop:u(t)1}
\end{prop}
\begin{proof}
For $n=0$ there is no cost, similar to $U_n(t)$. For $n\geq1$, we
have
\begin{align*}
\hat{U}_n(t)=U_{n-1}(0)h_n(t)dt+(1-h_n(t)dt)\hat{U}_n(t+dt)-Cdt+o(dt)
\end{align*}
where $h_n(t)$ is the corresponding hazard function of $X_n(t)$. The
differential equation is
\begin{align*}
\hat{U}_n'(t)=C-h_n(t)(\hat{U}_{n-1}(0)-\hat{U}_n(t))
\end{align*}
For both $U_n(t)$ and $\hat{U}_n(t)$ there are differential
equations with the same structure. Hence, their solutions also have
the same structure, where in this case, $X$ is replaced by $X_n(t)$.
\end{proof}
\begin{remark}
The fact that service time \replaced{has an}{follows} IMRL implies
that both $U_n(t)$ and $\hat{U}_n(t)$ are decreasing with $t$.
\end{remark}
\begin{remark}
the two expected utility are similar and they differ only by the
customers type. This is due to the fact that in both cases the
expected waiting time possesses the IMRL property.
\end{remark}

\subsection{Strategy profile}
As mentioned, customers' best response differs depending on the
number of customers in front of them and which of the two different
customer types they belong to. Therefore, each pair of queue length
and customer type needs to be considered separately. Also, customers
may balk, and clearly when the queue is long enough customers will
not join.  This happens when the expected utility is negative from
the moment one arrives. Hence, under the assumption of rationality
of customers, there is a maximum \replaced{number of customers
in}{length of} the system, and it is denoted by $n_{\max}$. We show
how to obtain it in section~\ref{section:findingnmax}.\\
\indent The customers' strategy profile is as follows. All customers
join if the system length is less than $n_{\max}$. Of course,
joining customers will adapt their strategy according to the
expected utility function, which is defined by the customer's type
and the number of customers in front of them in the queue. Both
$U_n(t)$ and $\hat{U}_n(t)$ are monotonically decreasing with $t$,
and therefore the best response is unique and hence pure. Let $T_n$
be the time a type \textbf{\romannum{1 }} customer who has $n$
customers in front of him in the system is willing to wait from the
moment of service completion until the next service completion
occurs, or abandonment is made by the customer in front of him. Let
$S_n$ be the time that a type \textbf{\romannum{2 }} customer who
has $n$ customers in front of him in the system is willing to wait
from his arrival point until a service completion occurs, or
abandonment is made by the customer in front of him. For any type of
customer if, while waiting, service completion occurs, he updates
his expected utility function and consequently his best response.
However if, while waiting, the customer in front of him abandons,
his best response is to abandon as well. This is because the
customer in front, who abandoned first, gathered more information
and his abandonment puts the customer behind in the exact same
position in the queue, to which the best response was to abandon.\\
\indent In conclusion, we have two \deleted{sets of} sequences that
define customers' strategies. The first one is
$\{T_1,T_2,...,T_{n_{\max}-2.}\}$  and the second is
$\{S_1,S_2,...,S_{n_{\max}-1}\}$. The largest index value of the
second sequence is $A_{n_{\max}-1}$, because it is the largest value
observed upon arrival for which one would be willing to join. That
is, if a customer joined and observed $n_{\max}$ customers in the
system, he would no longer join. In the first sequence, the largest
index is obtained when a customer in the $n_{\max}$th position in
the queue observes service completion. In this scenario, he has
${n_{\max}-2}$ others in front of him after the departure of the
customer who just completed his service. We next show how to obtain
the values of $T_n$ and $S_n$.
\section{Results} \label{section:result}
\subsection{Customers' best responses}
Using Propositions~\ref{prop:u(t)} and~\ref{prop:u(t)1} we show how
to obtain customers' best responses, for both type
\textbf{\romannum{1 }} and type \textbf{\romannum{2 }} customers,
given $N=n$. Specifically, we show how to obtain the sequences
$\{T_1,T_2,...,T_{n_{\max}-2.}\}$ and
$\{S_1,S_2,...,S_{n_{\max}-1}\}$. We begin with type
\textbf{\romannum{1 }} customers. Before presenting and proving the
following lemmas, we note that if the MRL is bounded it is possible
that finite thresholds are not possible. That is, even after waiting
a long time, one still benefit from staying. Thus, there is
condition on $\underset{t\rightarrow\infty}{\lim} m_X(t)$ in order
to compute the threshold values.
\begin{lem}\label{lem:Tn}
If $\underset{t\rightarrow \infty}{\lim}
m_X(t)>\frac{V}{C}-(n-1)\bar{x}$, and based on the expected utility
function from Proposition~\ref{prop:u(t)}, $T_n$ is the value of $t$
that solves
\begin{align}
G_n(t)=V-C(m_X(t)+(n-1)\bar{x})=0 \label{eq:Tn} \quad 1\leq
n\leq n_{\max}\!-\!2
\end{align}
Otherwise, $T_n=\infty$.
 \end{lem}
 \begin{proof}
Since $n<n_{\max}\!-\!2$, the expected utility for $t=0$ is positive
and of course, customers will be willing to wait as long their
expected utility is positive. Recall that the service time follows
IMRL and hence, the expected utility is monotone decreasing.
Moreover, the condition $\underset{t\rightarrow \infty}{\lim}
m_X(t)<\frac{V}{C}-(n-1)\bar{x}$, means that for some value of $t$
the expected utility will be negative. Therefore, by combining the
last two arguments, \replaced{E}{e}quation~(\ref{eq:Tn}) has a
unique and finite solution. Otherwise, the expected utility will be
positive for every $t>0$ and hence the best response is to stay
forever.
\end{proof}
\noindent Of course, the solution of equation~(\ref{eq:Tn}) is
straightforward, and hence $T_n$ for any possible $n$ can be easily
obtained.
\begin{lem}\label{lem:Sn}
If $\underset{t\rightarrow \infty}{\lim}\thinspace
m_{X_n(t)}\thinspace(t)>\frac{V}{C}-(n-1)\bar{x}$, and based on the
expected utility function from Proposition~\ref{prop:u(t)1}, $S_n$
is the value of $t$ that solves
\begin{align}
V-C(m_{X_n(t)}+(n-1)\bar{x})=0,\quad  1\leq n\leq
n_{\max}\!-\!1\label{eq:Sn}
\end{align}
Otherwise, $S_n=\infty$.
\end{lem}
\begin{proof}
We follow the same line of argument as Lemma~\ref{lem:Tn}.
\end{proof}
\noindent Solving~(\ref{eq:Sn}) is not straightforward, mainly
because obtaining the distribution of $X_n(t)$ is challenging. As
mentioned, $X_n(t) \,{\buildrel d \over
=}\,\{X\!-\!A(t)|N\!\!=\!\!n\}$. That is, in order to obtain the
distribution of $X_n(t)$, one must first obtain the distribution of
$A(t)|N\!\!=\!\!n$. Let $R(t,a)$ follow the distribution of the
residual service time, given that the service age upon arrival is
$a$ and $t$ time units have elapsed since the customer's arrival.
That is, $R(t,a) \,{\buildrel d \over
=}\,X\!-\!(t\!+\!a)|X\!>\!(a\!+\!t)$. Therefore, by using the law of
total probability, the following equation is equivalent
to~(\ref{eq:Sn}):

\begin{align}
              V-C\int_a(\mathbb{E}[R(t,a)|N=n]+(n-1)\bar{x})f_{A(t)|N=n}(a)da=0,\quad   n\leq n_{\max}\!-\!1\label{eq:baru(t)r}
\end{align}
Yet $f_{A(t)|N=n}(a)$ is unknown, and will be derived in the
following sections.
\subsection{Obtaining the maximum length of the queue}\label{section:findingnmax}
\begin{prop}\label{prop:nmax}
\begin{align}
n_{\max}= \sup \left\{ n\in \mathbb{N}:n
\leq\frac{\frac{V}{C}-\int_a\mathbb{E}[R(0,a)|N=n]f_{A(0)|N=n}(a)da+\bar{x}}{\bar{x}}\right\}+1
\label{eq:nmax}
\end{align}
\end{prop}

\begin{proof}
We seek the largest integer value of $n$ that obeys
$\hat{U}_n(0)>0$. Thus, after extracting $n$ from
equation~(\ref{eq:baru(t)r}) and applying the supremum we get the
expression in~(\ref{eq:nmax}).
\end{proof}
 We observe that once $f_{A(0)|N=n}(a)$ is derived, $n_{\max}$ can be computed.
\subsection{Markov chain underlying the process}\label{sect:markovchain}
Our motivation for using a Markov chain is mainly to obtain the pdf
of the service age for a tagged customer who observe\replaced{s}{d}
$n$ others in the \replaced{system}{queue} upon arrival, with $t$
time units having elapsed since then. This will eventually allow us
to find a strategy that holds for the Nash equilibrium.
\subsubsection{Markov chain state space}\label{markov chain}
First, we give some general notation for the steady states of the
Markov chain:
\begin{align*}
\mathcal{B}=\left\{(k,a,w_{k+1},w_{k+2},...,w_{n-1})\right\}
\end{align*}
where
\begin{itemize}
\item $k$ is the number of waiting customers that observed service completion;
\item $a$ is the age of the current service;
\item $w_i$ is the waiting time of the $i^{\mathrm{th}}$ customer in the system; and
\item $n$ is the number of customers in the system.
\end{itemize}
A general steady-state density is denoted by
$p(k,a,w_{k+1},w_{k+2},...,w_{n-1})$, and the probability density of
having $n$ customers in the system and a service age of $a$ is
denoted by $\pi(n,a)$. It can be derived from the steady states of
the Markov chain that
\begin{align*}
\pi(n,a)=\sum_{k=0}^{n-1}\int_{\underline{w}}p(k,a,w_{k+1},...,w_{n-1})d\underline{w},
\quad n\geq 1,\thinspace a\in \mathbb{R}^+, \end{align*} We denote
the marginal probability of having $n$ customers in the system by
$\pi_n$, which is derived as $\pi_n=\int_a \pi (n,a)da$.
We first indicate some general and rather trivial relationships.
From the arrival order, we get $w_{n-1}<w_{n-2}....< w_{k+1}<a$. Due
to \added{the} abandonment strategies, we get $a<T_k$ for $1\leq
k\leq n_{\max}\!-\!2$. This is because, if $a>T_k$, then the
$k^{\mathrm{th}}$ customer will have already abandoned by now. There
are no constraints on $a$ for $k=0$. That is, if no-one observed
service completion, the first in the queue could arrive at any value
of $a$. By the same argument,
 $$\thinspace w_{i} < S_{i},  \quad \mathrm{for}\ k+1\leq i\leq n-1. $$
For further analysis we present the following definition: let the
\emph{state structure} be the combination set of $(n,k)$, where $n$
is the number of customers in the system and $k$ is the number of
waiting customers that observed service completion. Since the queue
length is limited, the Markov \replaced{process}{chain} has a
limited number of different state structures.
\begin{prop}
The total amount of Markov chain state structures is
$$|\mathcal{B}_{(k,n)}|=\frac{n_{\max}(n_{\max}+1)}{2}$$\label{structure}
\end{prop}
\begin{proof}
We first claim that if there are $n$ customers in system, then there
are $n$ different state structures for $1\leq n\leq n_{\max}\!-\!1$.
What determines the \replaced{number}{amount} of state structures
for a given $n$, is the \replaced{number}{amount} of different
possible values of $k$, where $0\leq k\leq n\!-\!1$. That is, the
values of $k$ can be from zero to $n-1$, because even if everyone
observed service completion there would still be one in service.
However, for $n=n_{\max}$ there are $n_{\max}-1$ different state
structures. In this case, $0\leq k \leq n_{\max}\!-\!2$, while the
state $\{n_{\max}-1,a\}$ is not possible. Finally, $n=0$ means an
empty system, with just one state structure. The total number can be
computed as an arithmetic progression. It is equivalent to the sum
$\sum_{i=1}^{n_{\max}}i$, where each value of $i$ represents the
amount of state structures for a given $n$, except $i=n_{\max}$,
which in this case includes cases for both $n\!=\!0$ and
$n\!=\!n_{\max}$. From here the result is straightforward.
\end{proof}
\subsubsection{Finding the steady-state densities}
Due to the complexity of the process we begin with a simple example.
Let us consider \deleted{the} state $(0,a)$, which refers to an
active server with a current service age $a$ and an empty queue. For
$a<S_1$,
\begin{align}
p(0,a)=p(0,0)e^{-\l a}\bar{F}(a) \label{eq:p(0,a)1}
\end{align}
Equation~(\ref{eq:p(0,a)1}) justifies the following. State $(0,a)$
will always follow state $(0,0)$. This means that state $(0,0)$
occurred, and during the intervening $a$ time units there were no
service completions and no arrivals. Thus, the probability density
of $p(0,a)$ is as for $p(0,0)$, times the probability that there
were no service completions nor arrivals. However, for $a\!>\!S_1$,
we allow arrivals to occur from the beginning of service, assuming
that they will have abandoned by the time the service reaches age
$a$. For $m S_1\leq \!a\! < (m+1)S_1$, where $m\in \{0,1,2,3...\}$,
it is possible that a customer will arrive and abandon after $S_1$
time units. If, during the stay of the new arrival, more customers
arrive, then they will not be in the queue once he abandons. This is
because, if they stayed until then, they would abandon as well,
since the customer ahead of them abandoned. This process can happen
no more than $m$ times, for $a < (m+1)S_1$. For example, if $m=1$,
which means $S_1 < a < 2S_1$, then two cases are possible: no
arrivals at all and no service completion, or one arrival who
abandoned before the state reaches $(0,a)$ and no service
completion. From basic probability we obtain $$p(0,a)=p(0,0)( e^{-\l
a} +\l e^{-\l a}(a-S_1))\bar{F}(a)$$ Thus, $p(0,a)$ is equal to
$p(0,0)$ times the probability that no service completion occurs and
there were from $0$ to $m$ arrival events followed by abandonments.
We next give a general expression. Let $g(k,n,m,a)$ be
\begin{align*}
   g(k,n,m,a) = \begin{cases}
\left(\sum_{j=0}^m \frac{\l^je^{-\l(a-jS_{n})}(a-jS_{n})^j}{j!}\right)\overbar{F}(a)  & 1\leq k\!+\!1\!=\!n \leq n_{\max}\!-\!2\\
\left(\sum_{j=0}^m
\frac{\l^je^{-\l(w_{n-1}-jS_{n})}(w_{n-1}-jS_{n})^j}{j!}\right)
\frac{\overbar{F}(a)}{\overbar{F}(a-w_{n-1})}     &  1\leq k\!+\!1\!<\!n \leq n_{\max}\!-\!1   \\
\frac{\overbar{F}(a)}{\overbar{F}(a-w_{n-1})}       & \hspace{1.8cm} n=n_{\max}\\
                          \end{cases}
\end{align*}
and
\begin{align*}
    m\in  \begin{cases}
               \mathbb{N}_0       & k=0   \\
               \{0,1,2,... \left \lfloor{\frac{T_k-S_n}{S_n}}\right \rfloor\}      &1\leq k\leq n_{\max}-2
            \end{cases}
\end{align*}
where $a$ is the current service age, $w_{n-1}$ represents the
waiting time of the last joining type \textbf{\romannum{2 }}
customer who observed $n-1$ customers in the system upon arrival.
$n$ is the number of customers in the system, and $m$ relates to the
possible values of $a$: specifically, $mS_n<a<(m+1)S_n$. Lastly, $k$
is the number of customers who observed service completion.
\begin{lem}\label{lem:g(k,n,m,a)}
The function $g(k,n,m,a)$ is represented differently in three
cases.\\ Case 1, with $1\leq k\!+\!1\!=\!n \leq n_{\max}\!-\!2$,
represents the probability that the current state is $(k,a)$, given
that $a$ time units ago the state was $(k,0)$.
\\ Case 2, with $1\leq k\!+\!1\!<\!n \leq n_{\max}\!-\!1 $, represents the probability that the current state is now $(k,a,w_{k+1}...w_{n-2},w_{n-1})$, given that $w_{n-1}$ time units ago the state was $(k,a\!-\!w_{n\!-\!1},w_{k\!+\!1}\!-\!w_{n\!-\!1},...,w_{n\!-\!2}\!-\!w_{n\!-\!1},0)$.\\
Case 3 represents the probability that the current Markov state is
$(k,a,w_{k+1}...,w_{n_{\max}-2},w_{n_{\max}-1})$, given that
$w_{n_{\max}-1}$ time units ago the Markov state was:\\
$(k,a\!-\!w_{n_{\max}-1},w_{k\!+\!1}\!-\!w_{n_{\max}\!-\!1},...,w_{n_{\max}\!-\!2}\!-\!w_{n_{\max}\!-\!1},0)$.
\end{lem}
\begin{proof}
Case 1: in order that the state $(k,0)$ will be replaced by the
state $(k,a)$ after $a$ time units, we need to ensure that there
will not be a service completion during those $a$ time units. Also,
we need to ensure that there will not be any new arrivals, or if
there are, then they will have abandoned by the time the Markov
chain state reaches $(k,a)$. The probability of no service
completion is $\bar{F}(a)$. The probability of not having new
arrivals once the state reaches $(k,a)$ is $$\sum_{j=0}^m
\frac{\l^je^{-\l(a-jS_{n})}(a-jS_{n})^j}{j!}$$ Of course, scenarios
which include more than one abandonment made by the
$(n+1)^{\mathrm{th}}$ customer in the system are under
consideration, where $j$ is the number of times it occurs. We also
note that $j\leq m$. \\Case 2: in order that the state
$(k,a-w_{n-1},w_{k+1}-w_{n-1},...,w_{n-2}-w_{n-1},0)$ will be
replaced by the state $(k,a,w_{k+1}...w_{n-2},w_{n-1})$ after
$w_{n-1}$ time units, we need to ensure that there will not be
service completion and no new arrivals which stayed during that time
(similar to Case 1). The probability that there will not be service
completion is $\mathbb{P}(X>a|X>a-w_{n-1})$, which is equivalent to
$\frac{\bar{F}(a)}{\bar{F}(a-w_{n-1})}$. In Case 3, new arrivals are
not a possibility anyway, and therefore, in order that the state
will be transposed from
$(k,a\!-w_{n_{\max}\!-\!1},w_{k\!+\!1}\!-w_{n_{\max}\!-\!1},...,w_{n_{\max}\!-\!2}\!-\!w_{n_{\max}\!-\!1},0)$
to $(k,a,w_{k+1}...,w_{n_{\max}-2},w_{n_{\max}-1})$, we only need to
ensure that there will be no service completion, which is
$\frac{\bar{F}(a)}{\bar{F}(a-w_{n})}$ as in Case 2.
\end{proof}
Hence, from Lemma~\ref{lem:g(k,n,m,a)} we have, for $ m S_n \leq a
\leq (m+1)S_n $,
\begin{align}
p(k,a,w_{k\!+\!1},...,w_{n\!-\!1}\!)\!=\!p(k,a\!-\!w_{\!n\!-\!1},...,w_{n\!-\!2}\!-\!w_{\!n-\!1},0)g(k,n,m,a)
\label{eq:p(k,a,w)}
\end{align}
From~(\ref{eq:p(k,a,w)}) we \deleted{can}see that it is possible to
separate the expression of the steady-state densities of the Markov
chain into two parts. The first one is also a steady-state density
of the Markov chain, for which the last argument is set to be 0. The
second is the function $g(k,n,m,a)$, which is computable given the
model parameters. Therefore, in order to obtain the steady-state
densities of the Markov \replaced{process}{chain} we need to find
those in which the last argument is set to be 0. From \added{the}
balance equations,
\begin{align}
\l \pi_0=\int_{a}p(0,a)h(a)da\label{eq:steady1}
\end{align}
\begin{align}
\l
p(k,a,w_{k+1},...,w_{n-1})=p(k,a,w_{k+1},...,w_{n-1},0)\label{eq:steady2}
\end{align}
\begin{align}
\int_a
\sum_{i=0}^{n-1}p(i,a,w_{i+1},...,w_{n+1})h(a)da=p(n,0)\label{eq:steady3}
\end{align}
Recall from Proposition~\ref{structure} that the number of state
structures is $\frac{n_{\max}(n_{\max}+1)}{2}$. Excluding the state
$0$ for each state structure, there is \replaced{single}{one} state
for which the argument is $0$. Therefore, from
equations~(\ref{eq:steady1}) to~(\ref{eq:steady3}) we have
$\frac{n_{\max}(n_{\max}+1)}{2}$ equations. Where in fact
\deleted{(w.l.o.t)}from~(\ref{eq:steady1}) \replaced{consist}{there
is} one \replaced{E}{e}quation,~(\ref{eq:steady2})
\replaced{consist}{there are} $\frac{n_{\max}(n_{\max}-1)}{2}+2$
and~(\ref{eq:steady3}) \replaced{consist}{there are}
$n_{\max}\!-\!3$ equations. Combined with~(\ref{eq:p(k,a,w)}) and
the fact that $\sum_{n=0}^{n_{\max}}\pi_n=1$, all the steady states
of the Markov \replaced{process}{chain} \replaced{is}{can be}
derived.
\begin{remark}
For numerical computation\added{s}, we guess a value for $\pi_0$.
Using equations~(\ref{eq:p(k,a,w)}) to~(\ref{eq:steady3}), we
compute $\sum_{n=0}^{n_{\max}}\pi_n$. If the total sum is smaller
than 1, we guess a larger number for $\pi_0$ and vice versa.
\label{remark:numerical}
\end{remark}
\noindent We next give a special case where $n_{\max}\!=\!3$. There
are 6 different state structures, and they are represented as
\{$(0), (0,a),(0,a,w_1),(1,a),(1,a,w_2),(0,a,w_1,w_2)$\}. From
balance equations we state that\deleted{:}
\begin{align}
\l\pi_0=\int_{a=0}^\infty p(0,a)h(a)da \label{eq:p0ap0}
\end{align}
Also,
\begin{align}
\l p(0,a)= p(0,a,0) \label{eq:p0ap0w1}
\end{align}
\begin{align}
\l p(0,a,w_1)=p(0,a,w_1,0) \label{eq:p0aw1p0aw10}
\end{align}
\begin{align}
\int_{a=0}^{\infty}\int_{w_1=0}^{A_1\wedge a} p(0,a,w_1)h(a)da=
p(0,0) \label{eq:p0w2ap10}
\end{align}
\begin{align}
\int_{a=0}^{\infty}\int_{w_1=0}^{S_1\wedge a}\int_{w_2=0}^{S_2
\wedge w_1} p(0,a,w_1,w_2)h(a)da=p(1,0) \label{eq:p0aw1w2p1a0}
\end{align}
\begin{align}
\l p(1,a)=p(1,a,0) \label{eq:p1a0}
\end{align}
Using equations~(\ref{eq:p(k,a,w)}) and~(\ref{eq:p0ap0})
to~(\ref{eq:p1a0}), and applying the numerical procedure from
Remark~\ref{remark:numerical}, all steady states can be computed.
\subsection{The age distribution given the queue length} \label{section:fa}
We next show how to obtain $f_{A(t)|N=n}(a)$, while using the
\added{steady-}state probability densities of the Markov chain. Before doing so, we would like to emphasise it's complexity. Suppose a tagged customer arrived and observed one customer in the queue and one in service. This could result in many cases, for example, the system was empty, then one arrived and enter service and then another another arrival occurred. A different example, could be just like the previous one, only now, another customer arrived prior to the tagged customer's arrival and then abandoned. Of course, there can be many cases to consider and they get more complicated as the queue gets longer. For each scenario the age distribution will be computed differently. Although, it seems very complicated, we next show how via a few simple probability operations it can be done.\\ Let
$Y$ follow the distribution of the total amount of time a tagged
customer waited from arrival until either he abandons or there was
service completion, sampled by an outside inspector at an arbitrary
moment. In fact, $\{A(y)|N\!=\!n\} \thinspace \,{\buildrel d \over
=}\ \thinspace \{A|N\!=\!n,\!Y\!=\!y\}$. From Bayes' law,
\begin{align}
f_{A|N=n,Y=y}(a)=\frac{f_{Y|N=n,A=a}(y)f_{A|N=n}(a)}{\int_a
f_{Y|N=n,A=a}(y)f_{A|N=n}(a)da}.\label{eq:f_S_nx}
\end{align}
Since the presentation in a general case is implicit, we demonstrate
using a special case of $n_{\max}\!=\!3$. However, we can proceed
similarly for any value of $n_{\max}$. We show separately how to
obtain  $f_{A|N=n}(a)$ and $f_{Y|N=n,A=a}(y)$ for both $N=1$ and
$N=2$. $f_{A|N=n}(a)$ can be derived directly from the
\replaced{steady-state densities}{density steady state} of the
Markov \replaced{process}{chain}, specifically for $N=1$,
 $$f_{A|N=1}(a) =\frac{p(0,a)}{\pi_1}$$
and for $N=2$,
$$f_{A|N=2}(a) = \frac{p(1,a)+\int_{w_1=0}^{S_1 \bigwedge a}p(0,a,w_1)dw_1}{\pi_2}$$
We next derive $f_{Y|N=1,A=a}(y)$. Let $Q$ be a random variable that represents the inter-arrival times. Of course, $Q \sim \mathrm{Exp}(\l)$. Due to the PASTA property, an outside inspector sampling times is equivalent to customer arrival times. Therefore, $Q |Q \leq R(0,a) \wedge S_1$ is equivalent to $Y|N=\!\!1,A\!=\!a$.\\
\begin{lem}\label{lem:fYn1}
The conditional density of $Y$ given $A\!\!=\!\!a$, $N\!\!=\!\!1$
\deleted{and $Y<S_1\!\wedge\! R(0,a)$} is
\begin{align*}
    \frac{\l e^{-\l y}\frac{\bar{F}(a+y)}{\overbar{F}(a)}}{\mathbb{P}(Q \leq S_1 \wedge\ R(0,a))}
\end{align*}
and
\begin{align*}
\mathbb{P}(Q \leq S_1 \wedge\ R(0,a))=\int_{r=0}^{S_1}(1-e^{-\l
r})dr+\int_{r=S_1}^{\infty}(1-e^{-\l S_1})dr
\end{align*}
\end{lem}
\begin{proof}
\begin{align*}
\mathbb{P}(Q \leq y|Q \leq R(0,a)\wedge S_1) = \frac{\mathbb{P}(Q
\leq y,Q \leq R(0,a)\wedge S_1)}{\mathbb{P}(Q \leq R(0,a)\wedge
S_1)}
\end{align*}
We give explicit expressions for both numerator and denominator.\\
The numerator is
\begin{align*}
\mathbb{P}(Q\! \leq\! y,Q\! \leq R(0,a)\wedge\! S_1)=\int_{r=0}^y
\mathbb{P}(Q\!\leq\!
r)\frac{f(a+r)}{\overbar{F}(a)}dr+\int_{r=y}^{\infty}
\mathbb{P}(Q\!\leq\! y)\frac{f(a+r)}{\overbar{F}(a)}dr
\end{align*}
and the denominator is
\begin{align*}
\mathbb{P}(Q\leq R(0,a)\wedge S_1)=\int_{r=0}^{S_1}\mathbb{P}(Q \leq
r)\frac{f(a+r)}{\overbar{F}(a)}dr+\int_{r=S_1}^{\infty}\mathbb{P}(Q
\leq S_1)\frac{f(a+r)}{\overbar{F}(a)}dr
\end{align*}
After taking the derivative, we get the pdf of $Q|Q \leq
R(0,a)\wedge S_1$, and hence of $Y|N=\!\!1,A\!=\!a$ as well.
\end{proof}
\noindent We next derive  $f_{Y|N=2,A=a}(y)$. Let $W_1$ be a random
variable \replaced{taking}{that represents the} value of $w_1$ given
that the state is $(0,a,w_1)$, sampled at an arbitrary moment. The
pdf of $W_1$ is denoted by $f_{W_1}(w_1)$, and it is equivalent to
$\frac{p(0,a,w_1)}{\int_{u=0}^{S_1\bigwedge a}p(0,a,u)du}$. When a
customer arrives to a system given that $N=2$, there are two
possible state structures: \textbf{\romannum{1 }}. $p(1,a)$ and
\textbf{\romannum{2 }}.  $p(0,a,w_1)$. Let $I$ be an indicator that
receives the value of 1 when the state structure is $p(1,a)$, and
receives 0 when it is $p(0,a,w_1)$. From simple probability
considerations,
$\mathbb{P}(I=1)=\frac{p(1,a)}{p(1,a)+\int_{w_1=0}^{S_1\bigwedge a}
p(0,a,w_1)dw_1}$. We indicate that $Q|Q \leq R(0,a)\wedge S_2\wedge
(I(T_1-a)+(1-I)(S_1-W_1)) $ is equivalent to $Y|N\!=\!2,\!A=\!a$.
\begin{lem}\label{lem:fYn2}
The conditional density of $Y$ given $A=a$, $N=2$ and\\ $Y<
(I(T_1-a)+(1-I)(S_1-W_1))\wedge S_2\wedge R(0,a)$ is
\begin{align*}
&\mathbb{P}(I=1)\frac{\l e^{-\l
y}\frac{\bar{F}(a+y)}{\overbar{F}(a)}}{\mathbb{P}(Q \leq S_2 \wedge
R(0,a)\wedge (T_1-a))} +\\& \mathbb{P}(I=0)\int_{w_1=0}^{S_1}\!\!
\frac{\l e^{-\l y}\frac{\bar{F}(a+y)}{\overbar{F}(a)}}{\mathbb{P}(Q
\leq S_2 \wedge R(0,a)\wedge (S_1-w_1) )}p(0,a,w_1)dw_1
\end{align*}
and
\begin{align*}
&\mathbb{P}(Q \leq S_2 \wedge R(0,a)\wedge (T_1-a))=\\&
\int_{r=0}^{(T_1-a)\wedge A_2}\mathbb{P}(Q \leq
r)\frac{f(a+r)}{1-F(a)}dr+\int_{r=S_2\wedge(T_1-a)}^{\infty}
\mathbb{P}(Q \leq(T_1-a)\wedge S_2)\frac{f(a+r)}{\overbar{F}(a)}dr
\end{align*}
and
\begin{align*}
 &\mathbb{P}(Q \leq S_2 \wedge R(0,a)\wedge(S_1-w_1))= \\ & \int_{r=0}^{(S_1\!-\!w_1)\wedge S_2}\!\mathbb{P}(Q\! \leq \! r)\frac{f(a+r)}{\overbar{F}(a)}dr\!\! + \!\!\int_{r\!=\!(S_1\!-\!w_1)\! \wedge \! S_2}^{\infty}\!  \mathbb{P}(Q\!\leq\!(S_1\!-\!w_1)\wedge \!S_2)\frac{f(a+r)}{\overbar{F}(a)} dr
 \end{align*}
\end{lem}
The proof is given in Appendix A.
\subsection{Relations between the thresholds} \label{section:relations is the results}
In this section we refer to some trivial and nontrivial results
concerning the values of the sequences
$\{S_1,S_2,...,S_{n_{\max}\!-\!1}\}$ and
$\{T_1,T_2,...,T_{n_{\max}\!-2\!}\}$. Specifically, we describe the
dependencies and boundaries between them. Intuitively, the more
customers are in front of you, the less you would be willing to
wait.
\begin{lem}
If $T_i<\infty$,  $T_n> T_{n+1}$ for $1 \leq n \leq n_{\max}\!-\!3$.
\end{lem}
\begin{proof}
Recall that $t$ in equation~(\ref{eq:Tn}) refers to the time one
waited in the queue since service completion. Due to the fact that
$\mathbb{E}[X-t|X>t]$ is increasing with $t$, it follows
straightforwardly that the larger the value of $n$ the lower the
value of $t$ that solves the equation, and hence the lower the value
of $T_n$.
\end{proof}
\begin{lem}
If $S_n<\infty$,  $S_n > S_{n+1}$ for $1 \leq n \leq
n_{\max}\!-\!2$.\label{lem:Sn>Sn+1}
\end{lem}
\begin{proof}
Suppose an individual observed $n+1$ customers upon arrival. The
last event prior to his arrival could be either an arrival, an
abandonment or a service completion. If it \replaced{were}{was} an
arrival, then clearly the one who observed $n$ customers was in a
better situation when he arrived. If it was an abandonment, then
this individual (if he knew) should not join at all, and hence he is
in a worse situation than the one who observed $n$ upon arrival. The
last case is that where the arriving customer is the first to arrive
during the current service period. In this case, the age of the
service time is his inter-arrival time. Yet, we claim that the
greater the queue length, the smaller the probability of such an
event. This is because, after service completion when there are
either $n$ or $n+1$ customers in the system everything is the same
except one thing. In the $n+1$ system, the last customer can also
reach his abandonment time (which didn't exist in the $n$ customers
system). Thus, the probability that an arrival will occur prior to
service completion or abandonment is indeed lower in the $n+1$
customers system.
\end{proof}
\noindent We also observe that all $\{T_1,T_2,...,T_{n_{\max}-2}\}$
are obtained independently of everything else. This is a direct
result from equation~(\ref{eq:Tn}), where the value of $T_n$ is
determined by the values of the model parameter and $n$. Moreover,
we claim that $S_i$, for $\forall i\in \{1,2,...,n_{\max}\!-\!2\}$,
\deleted{is} depend\replaced{s}{ent} only on
$\{S_1,S_2,...,S_{i-1}\}$ and $\{T_1,T_2,...,T_i\}$, and for
$i=n_{\max}-1$, is dependent only on $\{S_1,...,S_{i-1}\}$ and
$\{T_1,T_2,...,T_{n_{\max}-2}\}$. This means, that one who observed
$i$ customers upon arrival, is not effected in any way, by actions
taken by customers who were in the $i+2^{th}$ position or worse in
the system (that is, have $i+1$ or more customers in from of them).
This result is rather surprising since at first thought it seems
that the entire history of a busy period would effect the value of
$S_i$. But the values of $S_i$ are not affected by
$\{S_{i+1},...,S_{n_{\max}-1}\}$ and
$\{T_{i+1},...T_{n_{\max}-2}\}$. The intuition behind this is as
follows. First, information regarding actions that took place in previous service period are irrelevant for the an arriving customer, because they have no impact (given the information he observes) on his waiting time.\ Now, let $C_i$ and $C_{i+1}$ be two customers who found $i$ and $i+1$ customers upon arrival, respectively, before an arrival a new customer, and all three arrived at the same service period.\ Also, assume that the new customer found $i$ customers in the system. Thus, $C_i$ and $C_{i+1}$ abandoned before the arrival of the new customer.\ The abandonment of $C_{i+1}$  could be triggered by either his loss of patience or by the loss of patience of $C_i$ .recall that if one abandons those after him abandon as well.\  The information that that the new customer faces is the same in both scenarios, because given that fact the $C_i$ abandoned, the threshold that $C_{i+1}$ has no influence on the state that the new customer faces.\

\begin{thm}\label{thm:sequence}
The Nash equilibrium profile is defined by two finite sequences of
thresholds, $\{T_1,T_2,...T_{n_{\max}-2}\}$ and
$\{S_1,S_2,...S_{n_{\max}-1}\}$, each sequence for each customer type. Within each sequence, from an individual point of
view, each threshold (e.g., the time he waits in line before reneging) is determined by the number of customers in
front of him and of course the model input.
\end{thm}
\begin{proof}
Customers' strategies are determined by their expected utility
functions. Based on Propositions~\ref{prop:u(t)}
and~\ref{prop:u(t)1}, we claim that there are two sequences because
the different customer types have different expected utility
functions, and therefore their strategies are different as well. The
sequence lengths are a direct result of Proposition~\ref{prop:nmax}
and the definition of the steady states of the Markov chain.
Finally, the fact that the Nash equilibrium profile within each
sequence is defined by thresholds was already proved in
Lemma~\ref{eq:Tn} and Lemma~\ref{eq:Sn}.
\end{proof}

\begin{remark}
After obtaining the steady state density probabilities (see
remark~\ref{remark:numerical}), we now use a different scheme to
compute the threshold values. First, the values of $T_1,
T_2,...,T_{n_{max}-2}$ are computed independently directly from
Equation~(\ref{eq:Tn}). Then, the values of $S_1,
S_2,...,S_{n_{max}-1}$ are computed recursively starting from $S_1$.
Thus, we guess a value of $S_1$ and while using
Equation~(\ref{eq:f_S_nx}) and employing the scheme in
Lemma~\ref{lem:fYn1} while using the steady state probabilities that
were obtained as described in Remark~\ref{remark:numerical}. If the
expected utility which is computed according to
Equation~(\ref{eq:Sn}) is negative we guess a lower value of $S_1$
and vice versa. After obtaining $S_1$ this goes on until we obtain
all values of $S_n$ for $n\in\{1,2,...,n_{max}-1\}$\footnote{We note
that although we provided in this paper the density probability of
$f_A|N=n$ for $n=1$ and $n=2$ only, by following the same line of
thought it can obtained for every $n\leq n_{max}-1$ only with a much
greater complexity. Even for $n=3$ (that is $n_{max}=4$) it gets
extremely difficult.}. The computations are very complex, which is
due to two main reasons. The first is that we don't have closed form
equations. As a result, there are numerous iterations for each value
that is being computed. The second is that the density probabilities
functions are very complicated and are calculated differently along
the support as a result of the function $g(k,n,m,a)$. Finally we
note that the both the procedure that was described in
Remark~\ref{remark:numerical} and in current one were computed by
using Wolfram-Mathematica software and we used a tolerance parameter
$\epsilon =10^{-3}$. \label{remark:numerical1}
\end{remark}

\subsection{Numerical result}
We present an example where the service distribution is
hyperexponential, and the model parameters are $V=4.85$, $c=1$,
$\mu_1=1$, $\mu_2=0.2$, $p=0.95$ and $\l=3$. The (symmetric) Nash
equilibrium is $n_{\max}=3$, $T_1=7.73$, $S_1=7.202$ and $S_2=3.13$.
In this example, the value of $S_2$ is significantly lower than
$S_1$. This is due to two reasons. The first is trivial, where one
needs to wait for an extra customer. The second is that the age
distribution is stochastically larger, which is a direct result from
Lemma~\ref{lem:Sn>Sn+1}.

\section{Conclusions} \label{section:discuusion}

In this study we show how to obtain the Nash equilibrium in an
observable $M/G/1$ queue with abandonments. We focus on service time
distributions which \replaced{have an}{follow} IMRL. The Nash
equilibrium is defined by two sequences of thresholds. The values of
the sequence $\{T_1,T_2,...T_{n_{\max}-2}\}$ represents the
abandonment thresholds for customers that observed service
completion, and they are obtained by solving a linear equation. The
values of the sequence $\{S_1,S_2,...,S_{n_{\max}-1}\}$, represents
the abandonment thresholds for customers that didn't observe service
completion, and obtaining them is much more difficult. They can be
computed recursively \replaced{based on the definition of the}{while
using a} Markov \replaced{process}{chain}. The reason we are able to
obtain them recursively is because customers' decisions are not
effected by \replaced{future arrivals}{those who are in a worse
positions in the queue}, (i.e. they are transparent to them). In
other words, from the point of view of a customer who is in the
$n^{\mathrm{th}}$ position in the queue, the maximum length of the
queue is $n$. Also, a numerical example is given in which both
sequences are computed. 
\\Finally, we discuss three limitations of our model: \\
\begin{enumerate}
\item We assumed that customers are homogenous with respect to their service value and waiting cost. Of course, considering heterogenous customers is more realistic. Perhaps, a future study may extend this results, but, it is vital that a simpler solution will be obtained first to build the foundation of such study.
\item We assumed that the time one already waited is considered to be a sunk cost and only future waiting time is considered. In real life, this may not always be the case. Relaxing this assumption, may have a huge impact on our solution due to the following reason: If one abandoned in front of me, I no longer necessarily abandon as well. This is because, I will not be in the exact situation as he was. This may alter the entire strategy profile.
    \item We assumed that the service distribution possess the IMRL property. The case where the service is with IFR was already argued for. We next discuss the case in which the hazard function is neither with IFR or DFR. If so, we believe that a general solution cannot be obtained due to the dependency on the service distribution.
\end{enumerate}
\bibliographystyle{plain}

\begin{appendix}
\section{Appendix A}\label{sec:appendixa}
\begin{proof}
\begin{align*}
&P(Q \leq y|Q \leq R(0,a)\wedge A_2\wedge(I(T_1-a)+(1-I)(A_1-W_1) )) \\ & = \frac{P(Q \leq y,Q \leq R(0,a)\wedge A_2\wedge(I(T_1-a)+(1-I)(A_1-W_1) ))}{P(Q \leq R(0,a)\wedge A_2\wedge(I(T_1-a)+(1-I)(A_1-W_1)))}
\end{align*}

We give explicit expressions for both the numerator and the denominator.\\
The numerator is
\begin{align*}
&P(Q \leq y,Q \leq R(0,a)\wedge A_2\wedge(I(T_1-a)+(1-I)(A_1-W_1) ))\\
&=\int_{w_1=0}^{A_1}  P(I=1)P(Q \leq y,Q \leq R(0,a)\wedge A_2\wedge(T_1-a))\\
&+P(I=0)\int_{a=0}^{A_2\wedge (A_1-w_1) }P(Q \leq y,Q \leq R(0,a)\wedge A_2\wedge(A_1-w_1) ) p(0,a,w_1)dw_1
\end{align*}
and the denominator is
\begin{align*}
&P(Q \leq R(0,a)\wedge A_2\wedge(I(T_1-a)+(1-I)(A_1-W_1)))\\
&=\int_{w_1=0}^{A_1} P(I=1)P(Q \leq R(0,a)\wedge A_2\wedge(T_1-a))\\
&+P(I=0)P(Q \leq R(0,a)\wedge A_2\wedge(A_1-w_1) ) p(0,a,w_1)dw_1
\end{align*}
After taking the derivative, we obtain the p.d.f.\ of $Q|Q \leq R(0,a)\wedge A_2\wedge(I(T_1-a)+(1-I)(A_1-W_1))$, and hence of $Y|N=\!\!2,A\!=\!a$ as well.

\end{proof}

\end{appendix}

\end{document}